\documentclass[journal,twoside,web]{ieee}

\usepackage{amsthm}

\usepackage{generic}
\usepackage{cite}
\usepackage{amsmath,amssymb,amsfonts}
\allowdisplaybreaks
\usepackage{algorithmic}
\usepackage{graphicx}
\usepackage{textcomp}
\usepackage{lipsum}
\usepackage{stfloats}
\usepackage{hyperref}
\PassOptionsToPackage{dvipsnames,svgnames,table}{xcolor}
\let\labelindent\relax
\usepackage{enumitem}  
\usepackage{microtype}
\usepackage{booktabs}
\usepackage{bigints}
\usepackage{mathrsfs}
\usepackage{mathtools}

\usepackage[norelsize,ruled,linesnumbered,algo2e]{algorithm2e}
\usepackage{standalone}

\usepackage{color}

\usepackage{multirow}
\usepackage{upgreek}
\usepackage{accents}

\newcommand{\ov}{\xbar}

\newcommand{\reals}{\mathbb{R}}

\newcommand{\Jf}{\mathfrak{I}}

\newcommand{\T}{\top}

\newtheorem{remark}{Remark}{}
\newtheorem{problem}{Problem}
\newtheorem{thm}{Theorem}
\newtheorem{cor}{Corollary}
\newtheorem{lem}{Lemma}
\newtheorem{defin}{Definition}

\newcommand*\xbar[1]{%
   \hbox{%
     \vbox{%
       \hrule height 0.7pt 
       \kern0.35ex
       \hbox{%
         \kern-0.1em
         \ensuremath{#1}%
         \kern-0.1em
       }%
     }%
   }%
} 

\begin{document}
\title{Static output feedback stabilization of uncertain rational nonlinear systems with input saturation}
\author{Thiago Alves Lima, Diego de S. Madeira, Valessa V. Viana, Ricardo C. L. F. Oliveira
\thanks{Thiago Alves Lima, Diego de S. Madeira, and Valessa V. Viana are with the Department of Electrical Engineering, Federal University of Ceará, Fortaleza, CE, Brazil. Ricardo C.L.F. Oliveira is with the School of Electrical and Computer Engineering, University of Campinas – UNICAMP, 13083-852, Campinas, SP, Brazil.}}

\maketitle

\begin{abstract}
In this paper, the notion of robust strict QSR-dissipativity is applied to solve the static output feedback control problem for a class of continuous-time nonlinear rational systems subject to input saturation and bounded parametric uncertainties. A local dissipativity condition is combined with generalized sector conditions to formulate the synthesis of a stabilizing controller in terms of linear matrix inequalities. The strategy applies to general static output feedback design without any restrictions on the plant output equation. An iterative algorithm based on linear matrix inequalities is proposed in order to compute the feedback gain matrix that maximizes the estimate of the closed-loop region of attraction. Numerical examples are provided to illustrate the applicability of this new approach in examples borrowed from the literature.
\end{abstract}


\section{Introduction}
\label{intro}

Saturating actuators are ubiquitous to real-world dynamical systems and are, by themselves, a nonlinearity to closed-loop systems that can degrade their performance or cause instability even when the open-loop system is modeled by linear methods \cite{Tarbouriech_2011}. Therefore, the consideration of such constraints is rather important when designing feedback control laws. Generally speaking, nonlinear models are closer to achieving the goal of well representing real systems than the much simpler linearized models around equilibrium points. However, their consideration increases the complexity of both open-loop analysis and closed-loop control design since, in contrast to the case of linear systems, most of the tools developed to nonlinear systems cannot be applied in a unified manner. Nonetheless, important contributions have been made in the last decades to provide useful tools to study such systems \cite{haddad2008nonlinear,Khalil_2002}. 

Some important works dealt with the analysis and control of rational nonlinear systems with input saturation. An early paper coping with this problem was \cite{coutinho_2010}, where analysis conditions based on linear matrix inequalities (LMIs) \cite{BEFB:94} were used for computing estimates of the region of attraction for rational control systems with saturating actuators. Later, in \cite{oliveira2012state}, state feedback design has been proposed for the special case of single-input systems. The design of this type of control law has also been tackled in \cite{azizi2018regional}, nonetheless with the consideration of multiple-input multiple-output (MIMO) systems with parametric uncertainties. More recently, the co-design of dynamical controllers and anti-windup loops was studied in \cite{Castro_2021}.

Nevertheless, none of these works have provided tools for the static output feedback (SOF) stabilization of rational nonlinear systems with saturating inputs. In fact, it is well known that a definitive solution for the SOF stabilization problem is not a consensus in the control community even for the case of linear systems (see the survey by \cite{SADABADI2016}), thus justifying the apparent scarcity of works for the much more complex case of nonlinear systems. Recently, the work in \cite{Madeira_2021} presented new necessary and sufficient conditions based on the notion of exponential QSR-dissipativity for the exponential stabilizability of nonlinear systems by linear static output feedback. Preliminary results regarding the development of a linear semidefinite programming (SDP) \cite{VB:96} approach for the robust asymptotic stabilization of nonlinear systems based on \cite{Madeira_2021} are available in \cite{madeira2020application}.

 QSR-dissipativity was first considered by \cite{willm2}, for the case of LTI systems. Soon after, it was applied for input-affine nonlinear systems by \cite{hill2} in their celebrated nonlinear version of the Kalman-Yakubovich-Popov (KYP) lemma. Then a principle called topological (graph) separation was introduced in \cite{safo1} as a geometrical tool for stability analysis of interconnected dynamical systems. An equivalence to Lyapunov stability conditions was claimed, although asymptotic stability was not addressed. Later on, robust stability analysis was tackled using the integral quadratic constraint (IQC) framework introduced by \cite{meg1} as a unifying approach for all previously mentioned methods. Nevertheless, controller design was not treated in that work. Then, a few decades after the publication of \cite{willm2}, \cite{hill2} and \cite{safo1}, topological separation was applied for establishing new necessary and sufficient conditions for linear SOF stabilization of LTI systems \cite{Peaucelle_2005}. The method consists in solving an LMI and a certain non-convex condition which provides a whole set of stabilizing controllers (resilient control). The linear part of the dissipativity-based approach of \cite{Madeira_2021} is equivalent to the results of \cite{Peaucelle_2005}, involving the same nonlinear inequality $SR^{-1}S^\top - Q\ge 0$. In \cite{Madeira_2021}, though, necessary  and sufficient stabilizability conditions for a broad class of nonlinear systems were provided, based instead on the constraint $SR^{-1}S^\top - Q= 0$.

In this paper, we explore the QSR-dissipativity results proposed in \cite{Madeira_2021} to address the problem of robust stabilization of the quite general class of rational nonlinear systems. Aiming the treatment of more realistic models, we consider that the plant is affected by input saturation and parametric uncertainties, thus increasing the level of complexity of the control design problem. We solve the stabilization problem by using a combination of QSR-dissipativity and generalized sector conditions. To address the most challenging issue, which is the formulation of the static output-feedback gain in terms of convex conditions, we propose a relaxation in the dissipativity condition using slack variables, solving the resulting inequalities in terms of an iterative algorithm, where LMIs are solved at each iteration. An algorithm is also formulated to design the stabilizing gain such that the closed-loop region of attraction is maximized. To the best of the authors' knowledge, this is the first work proposing a solution to the SOF stabilization problem of rational nonlinear systems with input saturation. The developed conditions can easily be used to tackle polynomial systems, as investigated in \cite{Valmorbida2013,RAN2016,JENNAWASIN2018}. Finally, numerical examples borrowed from the literature are presented to illustrate the effectiveness and less conservatism when compared (when possible) with previous approaches from the literature.

The rest of this paper is organized as follows. In Section~\ref{sec:pb}, the problem formulation and statement are presented, whereas the class of studied nonlinear systems is minutely presented. Theoretical preliminaries and the main results are presented in Sections~\ref{sec:preliminaries} and~\ref{sec:main}, respectively. Numerical examples from the literature are then used in Section~\ref{sec:simu}. Finally, Section~\ref{sec:conclu} closes the paper with a summary of conclusions and perspectives of future works.

\textbf{Notation.} For matrices $W = W^{\T}$ and $Z = Z^{\T} $ in $ \reals^{n \times n}$, $W \succ Z$ means that $W-Z$ is positive definite. Likewise, $W \succeq Z$ means that $W-Z$ is positive semi-definite. $\mathbb{S}_n^{+}$ stands for the set of $n \times n$ symmetric positive definite matrices. $I$ and $0$ denote identity and null matrices, respectively. The symbol $\star$ in the expression of a matrix denotes blocks induced by symmetry. For matrices $W$ and $Z$, \textit{diag}$(W,Z)$ corresponds to the block-diagonal matrix. The operator He$\{A\}$ denotes He$\{A\}=A+A^{\T}$. $f : \mathcal{X} \rightarrow \mathcal{Y}$ is a (vector) function with domain $\mathcal{X}$ and codomain $\mathcal{Y}$. $\mathcal{X} \times \mathcal{Y}$ is the Cartesian product of sets $\mathcal{X}$ and $\mathcal{Y}$. $sign(x)$ denotes the signum function, i.e, $sign(x) = -1$ if $x<0$, $sign(x) = 0$ if $x=0$, and $sign(x) = 1$ if $x>0$. Finally, $\mathcal{C}^1$ denotes the set of functions whose partial derivatives exist and are continuous, i.e., $f \in \mathcal{C}^1$ means that $f$ is continuously differentiable.  

\section{Problem formulation}
\label{sec:pb}

\subsection{System description}
Consider the following uncertain nonlinear system
\begin{equation}\label{eq:nonlinear:system}
    \begin{cases}
      \dot{x}(t) = f(x(t),\delta(t))+g(x(t),\delta(t))sat(v(t))\\
      y(t) = h(x(t),\delta(t))
    \end{cases}
\end{equation}
\noindent defined for $t\geq0$, where $x(t) \in \reals^{n}$ is the state with initial condition $x(0) \in \mathcal{X} \subseteq \reals^n$, $\delta(t) \in \mathcal{D} \subset \reals^l$ is a vector of bounded time-varying parameters which accounts for deviations of the model description around its nominal part and $y(t) \in \reals^{p}$ is the system output. To stabilize \eqref{eq:nonlinear:system}, we consider that
\begin{equation}\label{eq:control}
    v(t)=Ky(t)
\end{equation}
is a static output-feedback control law where $K \in \reals^{m \times p}$ is a matrix gain to be designed, and that the constrained plant input is given by the decentralized saturation function $sat(v(t)) \in \reals^{m}$ defined as
\begin{equation}\label{eq_sat}
sat(v_{(i)}) = sign(v_{(i)}) \min\{|v_{(i)}|,\xbar{u}_{(i)} \}, ~\xbar{u}_{(i)}>0, 
  \end{equation}
\noindent for $i=1, \dots, m$, where $\xbar{u}_{(i)}$ denotes the amplitude bound in each actuator. 

The set $\mathcal{X} \subseteq \reals^n$ is a given polytope (with $n_x$ vertices) of initial conditions $x(0)$ containing the origin for which one wants to study stability and stabilization of the closed-loop interconnection \eqref{eq:nonlinear:system}-\eqref{eq:control}. Such a polytope can be represented as the intersection of $n_{xe}$ hyperplanes \cite{coutinho_2010}
\begin{equation}\label{eq:setX}
    \mathcal{X} = \{x~|~a_k^{\T} x \leq 1, k=1,\dots, n_{xe}\},
\end{equation}
where the constant vectors $a_k \in \reals^{n}$ can be determined by fulfilling $a_k^{\T} x = 1$ at all groups of adjacent vertices of $\mathcal{X}$~\footnote{Set $\mathcal{X}$ does need to be positively invariant. Furthermore, a description of a set $\mathcal{D}$ similar to the one in \eqref{eq:setX} can also be obtained.}. Furthermore, functions $f(x,\delta)$, $g(x,\delta)$, and $h(x,\delta)$ are polynomial or rational functions on their arguments such that $(f,g,h) \in \mathcal{C}^1$, $\left(f(0,\delta),h(0,\delta)\right) = (0,0)$ for all $\delta \in \mathcal{D}$, and the origin $(x(t),u(t))=(0,0)$ is an equilibrium point of \eqref{eq:nonlinear:system}. Moreover, the computed control signal $v(t)$ is a measurable function, with $v(t) \in \mathcal{V} \subseteq \reals^{m}$ for all $t \geq 0$, $0 \in \mathcal{V}$. Then, the described system obeys the conditions for the existence and uniqueness of solutions $\forall x(0) \in \mathcal{X}$ (see, for example, \cite{haddad2008nonlinear}). 

\subsection{Differential Algebraic Representation -- DAR}
The dynamical system \eqref{eq:nonlinear:system} can be represented in many different and equivalent ways. In the case of a rational model, a much convenient representation is the well-known Differential Algebraic Representation (DAR), which is referred to as providing less conservative results than Linear Fractional Representations (LFR) and Linear Parameter Varying (LPV) forms. A DAR is more general than LFR and LPV approaches, and it usually leads to larger estimates of a domain of attraction \cite{coutinho_2010,azizi2018regional}. A DAR of an input-affine uncertain nonlinear system with input saturation such as \eqref{eq:nonlinear:system} is given by
\begin{align}
    \begin{cases}\label{eq:Plant}
      \dot{x}\hspace{-0.06cm}=\hspace{-0.06cm}A_1(x,\delta)x(t)\hspace{-0.06cm}+\hspace{-0.06cm}A_2(x,\delta)\pi(t)\hspace{-0.06cm}+\hspace{-0.06cm}A_3(x,\delta)sat(v(t))\\
      0\hspace{-0.06cm} =\hspace{-0.06cm} \Upsilon_1(x,\delta)x(t)\hspace{-0.06cm} +\hspace{-0.06cm} \Upsilon_2(x,\delta)\pi(t)\hspace{-0.06cm}+\hspace{-0.06cm}\Upsilon_3(x,\delta)sat(v(t))\\
      y \hspace{-0.06cm}=\hspace{-0.06cm} C_1x(t)+C_2\pi(t)
    \end{cases} 
    \end{align}
\noindent where $\pi(x,sat(v),\delta) \in \reals^{n_{\pi}}$ is a suitably chosen vector of nonlinear functions. Matrices $A_1(x,\delta) \in \reals^{n \times n}$, $A_2(x,\delta) \in \reals^{n \times n_{\pi}}$, $A_3(x,\delta) \in \reals^{n \times m}$, $\Upsilon_1(x,\delta) \in \reals^{n_{\pi} \times n}$, $\Upsilon_2(x,\delta) \in \reals^{n_{\pi} \times n_{\pi}}$, $\Upsilon_3(x,\delta) \in \reals^{n_{\pi} \times m}$ are affine functions with respect to $(x,\delta)$, while $\Upsilon_2(x,\delta) \in \reals^{n_{\pi} \times n_{\pi}}$ is supposed to be a square full-rank matrix for all vectors $(x,\delta) \in \mathcal{X} \times \mathcal{D}$, and $C_1, C_2$ are constant matrices of appropriate dimensions.  

The DAR of a system is not unique and the state-space representation \eqref{eq:nonlinear:system} is well-posed in its DAR form if $\Upsilon_2(x,\delta)$ is invertible, as from \eqref{eq:Plant} we have that 
\begin{equation*}
    \pi(x,sat(v),\delta) = -\Upsilon_2^{-1} \left(\Upsilon_1x(t)+\Upsilon_3 sat(v(t)) \right),
\end{equation*}
\noindent leading to
\begin{equation*}
    \dot{x} = \left(A_1-A_2 \Upsilon_2^{-1} \Upsilon_1 \right) x(t) + \left(A_3-A_2 \Upsilon_2^{-1} \Upsilon_3 \right) sat(v(t)).
\end{equation*}

\begin{remark}
The matrices $C_1$ and $C_2$ are constrained to be constant in order to derive numerically tractable design conditions. As a consequence, more complex vectors $\pi(x,sat(v),\delta)$ might have to be considered. Such choice is justified when developing the dissipativity-based stabilization conditions in the paper, since $C_1$ and $C_2$ appear quadratically in the conditions of the main theorem, which would require more complex relaxations (also potential sources of conservativeness) to check the parameter-dependent inequalities in the case of affine dependence on $(x,\delta)$.
\end{remark}

\subsection{Problem statement}

The problem solved in this letter can be summarized as follows.
\begin{problem}
Given the DAR matrices $A_1$, $A_2$, $A_3$, $\Upsilon_1$, $\Upsilon_2$, $\Upsilon_3$, $C_1$, and $C_2$ of the nonlinear system \eqref{eq:nonlinear:system}, polytopes $\mathcal{X}$, $\mathcal{D}$, and the saturation amplitude bounds $\xbar{u}_i$, $i=1,\dots,m$, develop conditions for the design of the static output feedback gain $K$ such that the asymptotic stability of the closed-loop system yielded from the connection \eqref{eq:nonlinear:system}-\eqref{eq:control} is ensured for some set of initial conditions $\mathcal{H} \subseteq \mathcal{X} \subset \reals^n$. Furthermore, estimates of the region of attraction of the saturated closed-loop system must be maximized.
\label{prob1}
\end{problem}

\section{Theoretical preliminaries}\label{sec:preliminaries}

\subsection{Dissipativity}
A dynamical system such as
\begin{align}
    \begin{cases}\label{nonlinearsystem}
\dot{x}(t)=f(x(t))+g(x(t),v(t)),\\
y(t)=h(x(t)),
    \end{cases}
\end{align}
is said to be dissipative if it is completely reachable and there exists a nonnegative storage function \(V(x(t))\), where \(V: \mathcal{X} \rightarrow \mathbb{R}\) and \(V \in \mathcal{C}^1\), and a locally integrable supply rate \(r(v(t),y(t))\) such that \(\dot{V} \leq r(v,y)\). In this work, we employ the definition of strict QSR-dissipativity presented below \cite{brogl1}.
\begin{defin}\label{defdissip}
A system is said to be strictly QSR-dissipative along all possible trajectories of \eqref{nonlinearsystem} starting at \(x(0)\), for all \(t \geq 0\), if there exists \(T(x)>0\) such that
\begin{equation}\label{dissipativityeq_geral}
 \dot{V}(x)+T(x) \leq y^{\top}Qy+2y^{\top}Sv+v^{\top}Rv,
\end{equation}
where \(S \in \mathbb{R}^{p \times m}\), \(Q \in \mathbb{R}^{p \times p}\), and \(R \in \mathbb{R}^{m \times m}\).
\end{defin}
 
\subsection{Finsler's Lemma}
In this subsection, the celebrated Finsler's Lemma is reproduced for convenience \cite{Mauricio_2001}.

\begin{lem} \label{finslerlemma}
Consider \(\mathcal{W} \subseteq \mathbb{R}^{n_s}\) a given polytopic set, and let \(\Phi:\mathcal{W} \rightarrow \mathbb{R}^{n_q \times n_q}\) and \(\Gamma:\mathcal{W} \rightarrow \mathbb{R}^{n_r \times n_q}\) be given matrix functions, with \(\Phi\) symmetric. Then, the following statements are equivalent 
\renewcommand{\theenumi}{\roman{enumi}}%
\begin{enumerate}
    \item \(\forall w \in \mathcal{W}\) the condition \(z^{\top}\Phi(w)z>0\) is satisfied \(\forall z \in \mathbb{R}^{n_q}:\Gamma(w)z=0\). 
    \item \(\forall w \in \mathcal{W}\) there exists a certain matrix function \(\Jf:\mathcal{W} \rightarrow \mathbb{R}^{n_q \times n_r}\) such that \(\Phi(w)+\Jf(w)\Gamma(w)+\Gamma(w)^{\top}\Jf(w)^{\top}\succ0\).
\end{enumerate}
\end{lem}
In order to obtain LMI conditions defined only at the vertices of the sets $\mathcal{X}$ and $\mathcal{D}$ in the case that \(\Gamma\) and \(\Phi\) are affine functions of \(w\), one can apply form \(ii)\) of Finsler's Lemma with a constant decision variable matrix \(\Jf\). In this case, \(ii)\) is only sufficient to guarantee \(i)\). 

\subsection{Generalized sector condition}

In the case of a system with input saturation, as the DAR \eqref{eq:Plant}, the satisfaction of \eqref{dissipativityeq_geral} might be hard to verify by means of SDP. In this paper, we tackle this problem by using the sector nonlinearity modelling approach, where the saturation function is replaced by an identity of the control $v(t)=f_v(x(t))$ and a deadzone nonlinearity (given in equation \eqref{eq:deadzoneident} below). Then, a sector condition is applied in order to help to verify \eqref{dissipativityeq_geral} for some set of the state space. The basis for the application of this approach is presented in the sequence. 

Consider the deadzone nonlinearity $\varphi$, defined as follows
\begin{equation}\label{eq:deadzoneident}
    \varphi(v(t)) = sat(v(t))-v(t),
\end{equation}
and the following set
\begin{equation}\label{polySet}
\begin{split}
         \mathcal{L}(\xbar{u}) =  \{ v \in \reals^{m}; \theta \in \reals^{m}; -\xbar{u} \leq v-\theta \leq \xbar{u} \}.
\end{split}
\end{equation}
\noindent where $\theta=f_{\theta}(x(t))$ is an auxiliary vector to be defined. We then recall the following Lemma from \cite[p.~43]{Tarbouriech_2011}.

\begin{lem}\label{lemma:Sec}
If $v$ and $\theta$ belong to set $\mathcal{L}(\xbar{u})$, then the deadzone nonlinearity $\varphi(v)$ satisfies the following inequality, which is true for any diagonal positive definite matrix $W \in \reals^{m \times m}$
\begin{equation}\label{SecBoun}
     \varphi^{\T}(v) W (\varphi(v)+\theta)\leq 0.
\end{equation}
\end{lem}

By taking into account the DAR \eqref{eq:Plant} and the identity \eqref{eq:deadzoneident}, the following equivalent representation is obtained
\begin{align}
    \begin{cases}\label{eq:Plantdeadzone}
      \dot{x}=A_1 x(t)+A_2 \pi(t)+A_3v(t)+A_3\varphi(v(t))\\
      0 = \Upsilon_1x(t) + \Upsilon_2\pi(t)+\Upsilon_3v(t)+\Upsilon_3\varphi(v(t))\\
      y = C_1 x(t)+C_2 \pi(t)
    \end{cases} 
    \end{align}
where the matrices dependency on $(x,\delta)$ have been omitted for sake of simplicity of notation and the control input $v(t)$ was defined in \eqref{eq:control}. Such representation allows us to analyze the system stability and dissipativity properties using a combination of \eqref{dissipativityeq_geral} and the generalized sector condition provided in Lemma~\ref{lemma:Sec}. Consider a partition of vector $\pi(t)$ as $\pi(t) = \begin{bmatrix} \pi_x (t)^{\T} & \pi_{h} (t)^{\T} \end{bmatrix}^{\T}$, where $\pi_x (t) \in \reals^{n_{\pi_{x}}}, n_{\pi_{x}} \leq n_{\pi}$, is a vector containing only terms of the state $x(t)$ such that the following equation can be written
\begin{equation}\label{eq:nullsig}
    0 = \Sigma_1 x(t) + \Sigma_2 \pi_x (t)
\end{equation}
with $\Sigma_1 \in \reals^{n_{\pi_{x}} \times n}$ and $\Sigma_2 \in \reals^{n_{\pi_{x}} \times n_{\pi_{x}}}$ being matrices with affine dependence on $(x,\delta)$. Next, consider the application of Lemma~\ref{lemma:Sec} with $\theta=f_{\theta}(x(t))=v-Gx-G_{\pi} \pi_x $, where the auxiliary matrices $G(x,\delta) \in \reals^{m \times n}$ and $G_{\pi}(x,\delta) \in \reals^{m \times  n_{\pi_{x}}}$ are affine in their arguments. Then, for the resulting set 
\begin{equation}\label{eq:setLu}
    \mathcal{L}(\xbar{u}) = \{x \in \reals^n; |G_{(i)}x+G_{\pi(i)}\pi_x | \leq \xbar{u}_{(i)},  {\scalebox{.9}{$i = 1,\dots,m$}} \}
\end{equation}
the inequality 
\begin{equation}\label{eq:seccond}
    B(\varphi,v,x,\pi)=-2\varphi^{\T}(v) W (\varphi^{\T}(v)+v-Gx- G_{aux}\pi ) \geq 0, 
\end{equation}
with $G_{aux} = \begin{bmatrix} G_{\pi} & 0_{m \times (n_{\pi}-n_{\pi_x})}\end{bmatrix}$ holds if $x \in \mathcal{L}(\xbar{u})$. Observe that $\mathcal{L}(\xbar{u})$ in \eqref{eq:setLu} is a set of vectors in the $x$ domain because $\pi_x $ is a function of $x$.  

\section{Main results}\label{sec:main}

The following theorem provides a solution to Problem~\ref{prob1}.
\begin{thm} \label{stabtheorem}
Assume that there exist matrices ${P} $ $\in$ $\mathbb{S}_{n}^{+}$, ${N} $ $\in$ $\mathbb{S}_{n}^{+}$, $R$ $\in$ $\mathbb{S}_{m}^{+}$, symmetric matrix $Q$ $\in$ $\reals^{p \times p}$, diagonal matrix $W$ $\in$ $\mathbb{S}_{m}^{+}$, matrices $S$ $\in$ $\reals^{p \times m}$, $\Jf$ $\in$ $\reals^{(n+n_{\pi}+2m) \times n_{\pi}}$, $L_s$ $\in$ $\reals^{(p+m)\times m}$, $Z$ $\in$ $\reals^{n_{\pi_x} \times n_{\pi_x}}$, and matrices $\xbar{G}(x,\delta)$ $\in$ $\reals^{m \times n}$, $\xbar{G}_{\pi}(x,\delta)$ $\in$ $\reals^{m \times n_{\pi_x}}$ with affine dependence on $(x,\delta)$, such that for all $(x,\delta)$ at the vertices of $\mathcal{X} \times \mathcal{D}$ 
\begin{equation} \label{main_matineq}
    {\Phi}+\Jf \Gamma+\Gamma^{\T}\Jf^{\T} \prec 0,
\end{equation}
\begin{equation}\label{eq:inclusiontset}
\begin{bmatrix}
    P & \Sigma_1^{\T} Z^{\T} & \xbar{G}_{(i)}^{\T} \\
    \star & \Sigma_2^{\T} Z^{\T} + Z \Sigma_2 & \xbar{G}_{\pi(i)}^{\T} \\
    \star & \star & 2W_{(i,i)}  -  \xbar{u}_{(i)}^{-2}
\end{bmatrix}\succeq 0, \quad i = 1,\dots,m
\end{equation}
\noindent and
\begin{equation}\label{main_stab}
    \begin{bmatrix}
Q & \star \\
S^{\T} & R 
\end{bmatrix}+ \mathrm{He}\{L_s\begin{bmatrix} S^{\T} & R \end{bmatrix}\} \prec 0,
\end{equation}
\begin{equation}\label{eq:inclusionpolytopic}
    \begin{bmatrix}
    P & a_k \\ 
    a_k^{\T} & 1
    \end{bmatrix} \succeq 0, \quad k=1, \dots, n_{xe},
\end{equation}
\noindent hold with $\Gamma =\begin{bmatrix} \Upsilon_1 & \Upsilon_2 & \Upsilon_3 & \Upsilon_3
\end{bmatrix}$ and matrix $\Phi$ given by
\begin{equation*}
   \Phi=\begin{bmatrix}
    \mathrm{He}\{P A_1\} + N -C_1^\T Q C_1 & \hspace{-0.5cm} \star & \hspace{-0.10cm} \star  & \hspace{-0.10cm} \star \\
    A_2^\T P - C_2^\T Q C_1 &  \hspace{-0.5cm} - C_2^\T Q C_2 & \hspace{-0.10cm} \star   & \hspace{-0.10cm} \star \\
    A_3^\T P-S^{\T}C_1 & \hspace{-0.5cm} -S^\T C_2 & \hspace{-0.10cm} -R & \hspace{-0.10cm} \star \\
    A_3^\T P + \xbar{G}  & \hspace{-0.5cm}  \begin{bmatrix} ~\xbar{G}_{\pi} & 0\end{bmatrix} & \hspace{-0.10cm} -W & \hspace{-0.10cm} -2W
    \end{bmatrix}.
\end{equation*}
Then, the SOF gain $K=-R^{-1}S^{\T}$ asymptotically stabilizes the closed-loop system \eqref{eq:nonlinear:system}-\eqref{eq:control} around the origin, and the ellipsoid $\varepsilon(P,1) \subset \mathcal{H} \subset \reals^n$, where $\mathcal{H}=\mathcal{X} \cap  \mathcal{L}(\xbar{u})$, with $\varepsilon(P,1) = \{x \in \reals^n; x^\T P x \leq 1 \}$ and $\mathcal{L}(\xbar{u})$ in \eqref{eq:setLu} with $G=W^{-1}\xbar{G}$, $G_{\pi}=W^{-1}\xbar{G}_{\pi}$, is an estimate of the closed-loop domain of attraction for all $\delta \in \mathcal{D}$. 
\end{thm}
\begin{proof}
First note that, by convexity, feasibility of \eqref{main_matineq} and \eqref{eq:inclusiontset} imply their satisfaction for all $(x,\delta) \in \mathcal{X} \times \mathcal{D}$. Then, consider relation \eqref{eq:inclusiontset}. Use the fact that 
\begin{equation}
    \left(\xbar{u}^{-2}_{(i)} -W_{(i.i)} \right)^{\T} \xbar{u}^2_{(i)} \left(\xbar{u}^{-2}_{(i)} -W_{(i,i)} \right) \geq 0
\end{equation}
\noindent to obtain the relation $2W_{(i,i)} -  \xbar{u}_{(i)}^{-2}\leq W_{(i,i)}^{\T} \xbar{u}_{(i)}^2 W_{(i,i)}$. Thus, satisfaction of \eqref{eq:inclusiontset} implies the fulfillment of 
\begin{equation*}
\begin{bmatrix}
    P & \Sigma_1^{\T} Z^{\T} & \xbar{G}_{(i)}^{\T} \\
    \star & \Sigma_2^{\T} Z^{\T} + Z \Sigma_2 & \xbar{G}_{\pi(i)}^{\T} \\
    \star & \star & W_{(i,i)}^{\T} \xbar{u}_{(i)}^2 W_{(i,i)}
\end{bmatrix}\succeq 0, \quad i = 1,\dots,m
\end{equation*}
Then, pre- and post-multiply the last inequality by $\textit{diag}\left(I,I,W_{(i,i)}^{-1}\right)$, apply a Schur complement, and use pre- and post-multiplication by $\begin{bmatrix}
 x^{\T} & \pi_x^{\T}
\end{bmatrix}^{\T}$ to obtain
\begin{align*}
    \label{elNotGood}
    ( G_{(i)}x &+G_{\pi(i)} \pi_x )^2~\xbar{u}^{-2} \leq \\ &x^{\T} P x + 2 \pi_x^{\T} Z (\Sigma_1 x + \Sigma_2 \pi_x), \quad i = 1,\dots,m
\end{align*}
\noindent which (by taking into account relation \eqref{eq:nullsig} and $x^\T P x \leq 1$) ensures the inclusion of the ellipsoid $\varepsilon(P,1)$ in the polyhedral set $\mathcal{L}(\xbar{u})$. ~Additionally, from relation \eqref{eq:inclusionpolytopic}, 
\begin{equation}
    x^{\T}Px-(x^{\T}a_k ) (a_k^{\T} x) \geq 0, \quad k = 1, \dots, n_{xe},
\end{equation}
which ensures $\varepsilon(P,1) \subset \mathcal{X}$. Therefore, by satisfying both \eqref{eq:inclusiontset} and \eqref{eq:inclusionpolytopic}, one ensures that $\varepsilon(P,1) \subset \mathcal{H}$, where $\mathcal{H}=\mathcal{X} \cap  \mathcal{L}(\xbar{u})$. Then, consider the quadratic Lyapunov function $V(x)= x^{\T} P x$, with ${P} $ in $\mathbb{S}_{n}^{+}$. If we can ensure that $\dot{V}(x)<0$ along the trajectories of the closed-loop system \eqref{eq:nonlinear:system}-\eqref{eq:control} for all initial conditions at the vertices of $\mathcal{W} = \mathcal{X} \times \mathcal{D}$, then it follows that $\varepsilon(P,1) \subset \mathcal{H}$ is an estimation on the domain of attraction, while $\mathcal{H}$ is a region of guaranteed asymptotic stability of the system. Consider \eqref{eq:seccond} and note that if
\begin{equation}\label{dissipativityeq_withsec}
 \begin{split}
     \dot{V}(x)+T(x)+B(\varphi,v,x,\pi) \leq 
      r(v(t),y(t))
 \end{split}
\end{equation}
is satisfied for matrices $(Q,S,R)$, then \eqref{dissipativityeq_geral} is also satisfied for all $x \in \mathcal{X} \cap \mathcal{L}(\xbar{u})$. By using $T(x)=x^{\T} N x$, ${N} $ in $\mathbb{S}_{n}^{+}$, the change of variables $\xbar{G}=WG$, $\xbar{G}_{\pi}=WG_{\pi}$, and substitutions of $y=C_1 x(t)+C_2 \pi(t)$ and $\dot{x}=A_1 x(t)+A_2 \pi(t)+A_3v(t)+A_3\varphi(v(t))$, \eqref{dissipativityeq_withsec} can be equivalently rewritten as $z^{\T} \Phi z \leq 0$, with extended vector $z = \begin{bmatrix}
x^{\T} & \pi^{\T} & v^{\T} & \varphi^{\T}(v)
\end{bmatrix}^{\T}$, and where $\Phi$ has been defined in Theorem~\ref{stabtheorem}. Clearly, matrix $\Gamma$ is a linear annihilator for the vector $z$. Then, from Lemma~\ref{finslerlemma}, $z^{\T} \Phi z \leq 0$ is satisfied for all $(x,\delta) \in \mathcal{H} \times \mathcal{D}$ if \eqref{main_matineq} is fulfilled at all vertices of the polytope $\mathcal{X} \times \mathcal{D}$ for some matrix $\Jf$. Furthermore, if at the same time we guarantee that 
\begin{equation}\label{eq:delta}
    r(v(t),y(t)) = y^{\top}Qy+2y^{\top}Sv+v^{\top}R v \leq 0,
\end{equation}
then $\dot{V}<0$ is also ensured for all $(x,\delta) \in \mathcal{H} \times \mathcal{D}$. By considering the extended vector $\zeta = \begin{bmatrix}
y^{\T} & v^{\T} 
\end{bmatrix}^{\T}$, inequality \eqref{eq:delta} can be equivalently rewritten as $\zeta^{\T} M_d \zeta \leq 0$, with 
\begin{equation*}
  M_d =  \begin{bmatrix}
Q & \star \\
S^{\T} & R 
\end{bmatrix}.
\end{equation*}
From \cite{Madeira_2021}, the control law \eqref{eq:control} with $K=-R^{-1}S^{\T}$ minimizes the supply rate in \eqref{eq:delta}. By noting that $C_s  \zeta = 0$, with $C_s = \begin{bmatrix} S^{\T} & R \end{bmatrix}$, Lemma~\ref{finslerlemma} can be applied. If there exists a matrix $L_s \in \reals^{(p+m)\times m}$ such that $M_d+ \text{He}\{L_s\begin{bmatrix} S^{\T} & R \end{bmatrix}\} \prec 0$, then $\zeta^{\T} M_d \zeta \leq 0$ is guaranteed for all $\zeta \neq 0$, leading to \eqref{main_stab} and completing the proof.\end{proof}
\begin{remark}\label{remark:pi} Consider the possible case when $\pi(t)=\pi_{h} (t)$, i.e., there is no $\pi_x(t)$ and matrices $\Sigma_1$ and $\Sigma_2$ satisfying \eqref{eq:nullsig}. Then, the sector condition is applied with the classical choice $\theta=f_{\theta}(x(t))=v-Gx$, with matrix $G$ being affine in $(x,\delta)$. In this case, matrix $\Phi$ and set $\mathcal{L}(\xbar{u})$ are given by the their simpler version considering $\xbar{G}_{\pi}=0$ and ${G}_{\pi}=0$, respectively. Furthermore, the inclusion condition \eqref{eq:inclusiontset} simplifies to
\begin{equation}\label{eq:inclusiontsetrem}
\begin{bmatrix}
    P & \xbar{G}_{(i)}^{\T} \\
    \star & 2W_{(i,i)}  -  \xbar{u}_{(i)}^{-2}
\end{bmatrix}\succeq 0, \quad i = 1,\dots,m.
\end{equation}
\end{remark}

\subsection{An iterative design procedure}

Conditions \eqref{main_matineq}, \eqref{eq:inclusiontset}, and \eqref{eq:inclusionpolytopic} in Theorem~\ref{stabtheorem} are LMIs and can be efficiently solved. On the other hand, condition \eqref{main_stab}, which ensures the negativity of $\dot{V}(x)$ (and therefore the closed-loop stability) contains a bilinearity due to the term $\text{He}\{L_s\begin{bmatrix} S^{\T} & R \end{bmatrix}\}$. Note that this bilinearity, ultimatelly, is not necessary and the procedure proposed in \cite{Madeira_2021} could be used to solve $\zeta^\T M_d \zeta \leq 0$. However, as proposed in \cite{Felipe_2020} in the context of SOF for uncertain linear systems, the inclusion of the slack variable $L_s$ through the Finsler's Lemma has been shown useful when solving the problem iteratively using the relaxation parameter $\lambda$. In this paper, we deal with this term by means of an iterative design procedure using a relaxed version of \eqref{main_stab} given by
\begin{equation}\label{main_stab:relaxed}
    \begin{bmatrix}
Q & \star \\
S^{\T} & R 
\end{bmatrix}+ \text{He}\{L_s\begin{bmatrix} S^{\T} & R \end{bmatrix}\} + \lambda \begin{bmatrix}
 -I_p & \star \\ 0 & 0 
\end{bmatrix} \prec 0,
\end{equation}
where $\lambda$ is an auxiliary scalar. 
\begin{remark}\label{remark:alg}
The multiplier $L_s$ in \eqref{main_stab} can be restrained to be of the form $L_s = \begin{bmatrix} -R^{-1}S^{\T} & -I_m \end{bmatrix}^{\T}$ without any conservatism. To see this, note that \eqref{main_stab} with this particular multiplier leads to
\begin{equation}
    \begin{bmatrix}
Q -\text{He}\{SR^{-1}S^{\T}\}& \star \\
-S^{\T} & -R 
\end{bmatrix}\prec 0
\end{equation}
while, by Schur complement, this is equivalent to $Q-S R^{-1} S^{\T} \preceq 0$, which is an alternative way to check the stability of the closed-loop with the SOF gain $K=-R^{-1}S^{\T}$ obtained by the direct substitution of $v=-R^{-1}S^{\T}y$ in \eqref{eq:delta} \cite{Madeira_2021}.  
\end{remark}
Remark~\ref{remark:alg} is interesting since it shows that when developing an iterative procedure based on the relaxed inequality \eqref{main_stab:relaxed}, the obtained values of $S$ and $R$ at each iteration can be mapped to $L_s$ at the next iteration. The remaining problem is the choice of how to initialize $L_s$ at the very first iteration. This problem is discussed in the sequence, after the design procedure is presented by Algorithm~\ref{Alg:1}.
\begin{algorithm2e}\label{Alg:1}
    \caption{Control design algorithm.}
\SetKwInOut{Input}{input}
    \SetKwInOut{Output}{output}
\Input{$i_{max}$}
    \Output{$K$, $R$, $S$, and $P$}
$i\gets 0$, $S_0 \gets 0$, and $R_0 \gets I$\; 
\While{$i<i_{max}$}{
    $L_s \gets \begin{bmatrix} -R_0^{-1}S_0^{\T} & -I_m \end{bmatrix}^{\T}$\;
    minimize $\lambda$ s.t. \eqref{main_matineq}, \eqref{eq:inclusiontset}, \eqref{eq:inclusionpolytopic}, \eqref{main_stab:relaxed}\;
\If{$\lambda \leq 0$ or $Q-S R^{-1} S^{\T} \preceq 0$}
{
  \Return $K=-R^{-1} S^{\T}$, $R$, $S$, and $P$\;
}
$i \gets i + 1$, $S_0 \gets S$, and $R_0 \gets R$\; 
}
\end{algorithm2e}

The Algorithm~\ref{Alg:1} is composed of two main parts. First, an initialization phase in line 1 takes place where the Finsler multiplier is initialized with matrices $S_0=0$ and $R_0=I$. Then, matrices $S_0$ and $R_0$ are used to update the multiplier $L_s$ in a while loop that searches for a solution to \eqref{eq:delta} by means of the relaxed inequality \eqref{main_stab:relaxed}. Note that the gain $-R_0^{-1} S_0^{\T}$ does not need to be a stabilizing one since the LMI \eqref{main_stab:relaxed} is feasible thanks to the relaxation variable $\lambda$. The solution $K=-R^{-1} S^{\T}$ is then valid if either $\lambda\leq0$ or $Q-S R^{-1} S^{\T} \preceq 0$, since both ensure the fulfilment of \eqref{eq:delta}. Note that it is possible that \eqref{main_stab} be satisfied even with positive values of $\lambda$, therefore the verification performed in line 5 can be helpful to decrease the number of iterations in Algorithm~\ref{Alg:1}. 

\begin{thm}\label{teo:alg}
The inequality \eqref{main_stab:relaxed} is always feasible at the first iteration with any initializing choices of $S_0$ and $R_0$. Furthermore, at each following iteration in the while loop, the objective $\lambda$ is nonincreasing.
\end{thm}
\begin{proof}
Suppose that at the first iteration, LMIs \eqref{main_matineq}, \eqref{eq:inclusiontset}, and \eqref{eq:inclusionpolytopic} hold for a solution set of matrices $\mathcal{Y}_{1}=\{P_{1},N_{1},R_{1},Q_{1},W_{1},S_{1},\Jf_{1},\xbar{G}_{1},\xbar{G}_{\pi1},Z_{1}\}$. Inequality \eqref{main_stab:relaxed} is also feasible since $L_s = \begin{bmatrix} -R_0^{-1}S_0^{\T} & -I_m \end{bmatrix}^{\T}$ leads to the condition
\begin{equation*}
    \begin{bmatrix}
Q-\lambda_1 I_p -\text{He}\{S R_0^{-1} S_0^{\T}\} & \star \\
-R R_0^{-1}S_0^{\T} & -R 
\end{bmatrix}\prec 0,
\end{equation*}
which can always be satisfied with large enough $\lambda_1$ since $R$ is a positive definite matrix. Then, due to the structure of inequality \eqref{main_stab:relaxed} derived from Finsler's lemma, at the next iteration there exists large enough $\lambda_2$ and a set of matrices $\mathcal{Y}_{2}$ satisfying the problem in line 4 since this can be achieved at least with the trivial solution $\mathcal{Y}_2=\mathcal{Y}_{1}$, $\lambda_2=\lambda_1$. For each following iteration, the same logic applies, where there exists at least the trivial solution $\lambda_{i+1}$ = $\lambda_{i}$, $\mathcal{Y}_{i+1}=\mathcal{Y}_{i}$, meaning that the new $\lambda$ is at least as good as the one from the previous iteration.   
\end{proof}

As a final contribution, an optimization problem can also be formulated to design a static output feedback gain that stabilizes the closed-system \eqref{eq:nonlinear:system}-\eqref{eq:control} while maximizing the estimated region of attraction $\varepsilon(P,1)$. When maximizing such region, different criteria can be adopted, such as maximization of the ellipsoid minor axis, volume maximization, maximization in desired directions, etc. In this work, we adopted volume maximization, which is equivalent to minimization of the trace of matrix $P$. To this end, the following algorithm takes place, where $\gamma$ is a small number used to stop the algorithm when a near-optimal solution to $trace(P)$ is found.

\begin{algorithm2e}\label{Alg:2}
\SetKwInOut{Input}{input}
    \SetKwInOut{Output}{output}
\Input{$i_{max}$, $\gamma$, and \{$R$,$S$,$P$\} solution to Algorithm~\ref{Alg:1}.}
    \Output{$K$ and $P$}
    \caption{Maximization of $\varepsilon(P,1)$.}
$i\gets 0$, $S_0 \gets S$, $R_0 \gets R$, and $P_0 \gets P$\;
\While{$i<i_{max}$}{
    $L_s \gets \begin{bmatrix} -R_0^{-1}S_0^{\T} & -I_m \end{bmatrix}^{\T}$\;
    minimize $trace(P)$ s.t. \eqref{main_matineq}, \eqref{eq:inclusiontset},
    \eqref{main_stab}, \eqref{eq:inclusionpolytopic}\;
\If{$|trace(P) - trace(P_0)| \leq \gamma$}
{
  \Return $K=-R^{-1} S^{\T}$ and $P$\;
}
$i \gets i + 1$, $S_0 \gets S$, $R_0 \gets R$, and $P_0 \gets P$\; 
}
\end{algorithm2e}

\begin{cor}\label{cor:max}
For a given $\gamma>0$, Algorithm~\ref{Alg:2} always returns an output for a sufficient amount of iterations $i_{max}$.
\end{cor}
\begin{proof}
Given any solution set of matrices  $\mathcal{Y}_{s}=\{P_{s},N_{s},R_{s},Q_{s},W_{s},S_{s},\Jf_{s},\xbar{G}_{s},\xbar{G}_{\pi s},Z_{s}\}$ obtained with Algorithm~\ref{Alg:1}, there exists at least one feasible solution to Algorithm~\ref{Alg:2} at iteration $i=0$ given by the same set of matrices $\mathcal{Y}_{0}=\mathcal{Y}_{s}$. To see this, note that satisfaction of \eqref{main_matineq}, \eqref{eq:inclusiontset}, \eqref{eq:inclusionpolytopic} with this trivial solution is guaranteed. Then, by Schur complement, one obtains that \eqref{main_stab} is equivalent to $R_s^{-1}>0$ and $Q_s-S_s R_s^{-1} S_s^{\T} \preceq 0$, which is also certified since $\mathcal{Y}_{s}$ is a solution to Algorithm~\ref{Alg:1}. The same reasoning can be applied to the subsequent steps, where there always exists at least the solution $\mathcal{Y}_{i+1}$ = $\mathcal{Y}_{i}$, which leads to $|trace(P_{i+1}) - trace(P_i)|=0<\gamma, ~\forall \gamma > 0$. 
\end{proof}
From Corollary~\ref{cor:max} and its proof, it is clear that at each iteration in the while loop of Algorithm~\ref{Alg:2} there exists a control gain $K_{i+1}$ leading to an ellipsoid $\varepsilon(P,1)$ at least as large as the one obtained at the previous iteration with control $K_{i}$. Throughout this paper, we always consider $\gamma = 10^{-2}$. Smaller values of this parameter can lead to more refined estimations on the set $\varepsilon(P,1)$ at the cost of more iterations in Algorithm~\ref{Alg:2}. 


\section{Simulation results}\label{sec:simu}

In this section, we present three numerical examples to illustrate the use of the proposed solution. Initially, a comparison with an iterative method from the literature to design scheduled static output feedback controllers for polynomial systems is presented. Then, a second example dealing with an uncertain rational nonlinear system from the literature is detailed. Then, the final example consists of a comparison for the case of MIMO polynomial nonlinear systems. The results were obtained programming the proposed conditions in Matlab 2019b using the YALMIP parser \cite{Lof:04} and the SDP solver MOSEK \cite{AA:00} release 9.1.11, in a PC equipped with: Core i7-4500U (1.80 GHz, 64 bits), 8 GB of RAM, Linux Mint 19.3. 

\subsection{Example 1 - Comparison with iterative method}

Consider the input saturated polynomial system analysed in Example 1 from \cite{jenna2021}. A DAR of this system is given by 
\begin{equation*}
    \begin{split}
&\pi=\pi_x=\begin{bmatrix}
x_1^2 & x_2^2
\end{bmatrix}^\T,~\Upsilon_2 = -I_2,~A_3 = \begin{bmatrix}
0 & 1
\end{bmatrix}^\T,\\
&A_1=\begin{bmatrix}
-1 & \frac{1}{4}\\
0 & 0
\end{bmatrix},~\Upsilon_3 = \begin{bmatrix}
  0 \\ 0
  \end{bmatrix},~C_2=\begin{bmatrix}
   0 \\ 0
  \end{bmatrix}^\T,~C_1=\begin{bmatrix}
   1 \\ -1
  \end{bmatrix}^\T,\\
  &\Upsilon_1 =
  \begin{bmatrix}
 x_1 & 0\\
 0  & x_2
  \end{bmatrix},~A_2=\begin{bmatrix}
1-\frac{3}{2}x_1 - x_2 & -\frac{3}{4}x_1 - \frac{1}{2}x_2\\ 
 0 & 0
\end{bmatrix}.
    \end{split}
\end{equation*}
In this case, relation \eqref{eq:nullsig} is satisfied with matrices $\Sigma_1=-\Upsilon_1$ and $\Sigma_2=I_2$. \cite{jenna2021} uses a rational Lyapunov function \(V(x)=x^\T P(x)^{-1} x\) to develop iterative conditions for the design of a polynomial output feedback gain $K(y)$ that stabilizes this system. Consider \(\ov{u}=1.5\), which is the strictest input saturation value used in \cite{jenna2021}. By running Algorithms~\ref{Alg:1} and~\ref{Alg:2} with polytope $\mathcal{X}$ defined with the limits $|x_1|\leq 0.9,~|x_2|\leq 0.9$ we obtain, after a total of eight iterations, the stabilizing gain $K=0.3785$ and the results summarized in Table~\ref{table:ex1}. 

As seen in Table~\ref{table:ex1}, the conditions from \cite{jenna2021} are infeasible when the degree of \(P(x)\) is lower than two. It is important to highlight that the proposed approach provided the best results even in comparison with the obtained by \cite{jenna2021} using $P(x)$ with a polynomial dependence of degree 3. Observe that Table~\ref{table:ex1} shows the maximum radius of the domain of attraction, which is the metric presented in \cite{jenna2021}. Nonetheless, even the value of the semi-minor axis of the ellipsoid obtained with the proposed method, also given by 0.8999, is bigger than the maximum radius obtained for the domain of attraction in \cite{jenna2021}. This shows the benefit of our strategy, which synthesized a stabilizing static output feedback gain in contrast to \cite{jenna2021} which utilised the more complex scheduled feedback gain $K(y)$. It is important to highlight that the number of iterations of the algorithm from \cite{jenna2021} was not provided, so no comparison regarding this aspect could be made.
\begin{table}
\centering
\caption{Example 1 - Comparison of the maximum radius of the domain of attraction for $\ov{u}=1.5$. $n$ stands for the degree of the Lyapunov matrix $P(x)$ in \cite{jenna2021}.}
\label{table:ex1}
\begin{tabular}{ccccc} 
\hline
\multirow{2}{*}{Proposed} & \multicolumn{4}{c}{\cite{jenna2021}}  \\ 
\cline{2-5} & $n=0$ & $n=1$ & $n=2$ & $n=3$ \\ 
\hline
0.9001 & Infeasible & Infeasible & \multicolumn{1}{l}{0.5711} & \multicolumn{1}{l}{0.5816}  \\
\hline
\end{tabular}
\end{table}



\subsection{Example 2 - Rational nonlinear system}
Consider the uncertain rational nonlinear input saturated plant analysed in Example 5.4 from \cite{azizi2018regional}, which corresponds to an inverted pendulum system. The pendulum parameters and the DAR for the system are the same considered in \cite{azizi2018regional}, which is such that $\pi(t)=\pi_h(t)$, meaning that Remark~\ref{remark:pi} applies. Also, we consider the same saturation limit \(\xbar{u}=0.25\). As we did not find any papers dealing with SOF design for rational nonlinear systems, we consider a comparison with \cite{azizi2018regional}, which designs static state feedback (SSF) gains. For a fair comparison, we consider the same state and uncertainty polytope limits as \cite{azizi2018regional}, with $|\delta_1| \leq 0.10$, $|\delta_2| \leq 0.99$, $|x_1| \leq 0.19$, and $|x_2| \leq 0.22$. First, we design a static state feedback (SSF) (equivalent to having an output $y(t)=x(t)$), by running Algorithms~\ref{Alg:1} and~\ref{Alg:2} to obtain the stabilizing gain $K = \begin{bmatrix}
  -1.9791 &  -3.3947 \end{bmatrix}$. Then, to illustrate the ability of the proposed method to deal with SOF design, we consider an output feedback design for $y(t) = \begin{bmatrix} 1 & 1  \end{bmatrix}x(t)$. In this case, we obtain the stabilizing gain $K=-2.3968$. For comparison, Table~\ref{results_table} summarizes the obtained results with respect to the $log(det(P^{-1}))$, as in \cite{azizi2018regional}. Note that both the proposed SSF and SOF lead to larger estimates on the closed-loop region of attraction since greater values of $log(det(P^{-1}))$ are obtained\footnote{The volume of an $n$-dimensional ellipsoid $\varepsilon(P,1)$ is proportional to $log(det(P^{-1}))$ \cite[p. 60]{Tarbouriech_2011}.}. To have a more intuitive comparative measure, the area of the ellipsoid $\varepsilon(P,1)$ is also computed in Table~\ref{results_table}.
\begin{table}[h!]
\centering  
\caption{Example 2 - Summary of results.}
\label{results_table}
\begin{tabular} {clclclc} 
\hline
\multicolumn{1}{l}{}  & \hspace{-0.4cm} \scriptsize \cite{azizi2018regional}  & \scriptsize Proposed - SSF \& SOF \\  
\hline
\hspace{-0.2cm}$log(det(P^{-1}))$  & $~~-6.800$   & $-6.789$  \& $-6.726
$  \\ 
\hspace{-0.2cm}Area of $\varepsilon(P,1)$  & $~~~~0.1046$   & $~0.1054$ \& $~0.1088
$  \\ \hline
\end{tabular}
\end{table}
\subsection{Example 3 - MIMO polynomial system}
Consider the MIMO uncertain input saturated polynomial system in its DAR form analysed in Example~5.3 from \cite{azizi2018regional}. The DAR is such that $\pi(t)=\pi_x(t)$, with relation \eqref{eq:nullsig} being satisfied with $\Sigma_1 = \begin{bmatrix}
 -x_1 & 0
\end{bmatrix}$ and $\Sigma_2 = 1$. For sake of comparison, we consider $y(t)=x(t)$ and $\ov{u}=[1~1]^\T$, which are the same output equation and saturation bounds considered in \cite{azizi2018regional}, respectively. Let us assume $|x_1|\leq 4$ and $|\delta| \leq 0.8$, which correspond to the $\mathcal{X}$ polytope and the largest uncertainty used in \cite{azizi2018regional}. Then, by running Algorithm~\ref{Alg:2} we obtain a stabilizing gain given by
\begin{equation*}
    K=\begin{bmatrix}
   -3.1921 &  -0.4858 \\
    -21.4646  & -3.6065
    \end{bmatrix}
\end{equation*}
and a matrix $P$ such that $log(det(P^{-1}))=8.4600$, which means we obtain a much larger estimate on the closed-loop region of attraction than the one in \cite{azizi2018regional} since it can be checked in Figure 6 from \cite{azizi2018regional} that the approach therein obtains $log(det(P^{-1}))\leq 4$. Now, to demonstrate the ability of the proposed technique to deal with nonlinear output maps, consider the same state and uncertainty polytopes but with $y=h(x,\delta) =  x_1 + x_2 + 0.5x_1^2$, which can be dealt with by using DAR matrices $C_1 = \begin{bmatrix} 1 & 1 \end{bmatrix}$ and $C_2 = 0.5$. In this case, we obtain a gain $K=\begin{bmatrix}-0.0746 & -1.6825\end{bmatrix}^{\T}$ and $log(det(P^{-1}))= 6.6871$, which is also larger than the estimate from \cite{azizi2018regional}. 
 
\section{Conclusion}\label{sec:conclu}

This work proposed new conditions and an associated iterative algorithm for the design of stabilizing static output feedback gains for nonlinear systems affected by both uncertainties and input saturation. The approach is aimed at rational nonlinear systems. However, as demonstrated in the paper, it also easily applies to the less general case of polynomial nonlinear systems. An optimization procedure leading to feedback gains that maximize the closed-loop region of attraction was also formulated. The application of the proposed method to nonlinear systems from the literature demonstrated the effectiveness of the iterative design algorithms in terms of the obtained estimations on the closed-loop region of attraction. 

Ongoing research includes the application of equilibrium-independent dissipativity for feedback stabilization of nonzero equilibria using linear SDP strategies. Also, future work will envisage the use of rational Lyapunov functions in order to decrease conservatism of the conditions \cite{Trofino_2014}.



\bibliographystyle{IEEEtran}
\bibliography{refs}

\begin{thebibliography}{10}
\providecommand{\url}[1]{#1}
\csname url@samestyle\endcsname
\providecommand{\newblock}{\relax}
\providecommand{\bibinfo}[2]{#2}
\providecommand{\BIBentrySTDinterwordspacing}{\spaceskip=0pt\relax}
\providecommand{\BIBentryALTinterwordstretchfactor}{4}
\providecommand{\BIBentryALTinterwordspacing}{\spaceskip=\fontdimen2\font plus
\BIBentryALTinterwordstretchfactor\fontdimen3\font minus
  \fontdimen4\font\relax}
\providecommand{\BIBforeignlanguage}[2]{{%
\expandafter\ifx\csname l@#1\endcsname\relax
\typeout{** WARNING: IEEEtran.bst: No hyphenation pattern has been}%
\typeout{** loaded for the language `#1'. Using the pattern for}%
\typeout{** the default language instead.}%
\else
\language=\csname l@#1\endcsname
\fi
#2}}
\providecommand{\BIBdecl}{\relax}
\BIBdecl

\bibitem{Tarbouriech_2011}
S.~Tarbouriech, G.~Garcia, J.~M. {Gomes da Silva Jr.}, and I.~Queinnec,
  \emph{Stability and Stabilization of Linear Systems with Saturating
  Actuators}.\hskip 1em plus 0.5em minus 0.4em\relax London: Springer, 2011.

\bibitem{haddad2008nonlinear}
W.~M. Haddad and V.~Chellaboina, \emph{Nonlinear dynamical systems and control:
  a {L}yapunov-based approach}.\hskip 1em plus 0.5em minus 0.4em\relax
  Princeton, NJ, USA: Princeton university press, 2008.

\bibitem{Khalil_2002}
H.~K. Khalil, \emph{Nonlinear Systems}.\hskip 1em plus 0.5em minus 0.4em\relax
  Upper Saddle River, NJ, USA: Prentice Hall, 2002.

\bibitem{coutinho_2010}
D.~F. Coutinho and J.~M. {Gomes da Silva Jr.},
  ``\BIBforeignlanguage{English}{Computing estimates of the region of
  attraction for rational control systems with saturating actuators},''
  \emph{\BIBforeignlanguage{English}{IET Control Theory \& Applications}},
  vol.~4, pp. 315--325(10), March 2010.

\bibitem{BEFB:94}
S.~Boyd, L.~El{ }Ghaoui, E.~Feron, and V.~Balakrishnan, \emph{Linear Matrix
  Inequalities in System and Control Theory}.\hskip 1em plus 0.5em minus
  0.4em\relax Philadelphia, PA: SIAM Studies in Applied Mathematics, 1994.

\bibitem{oliveira2012state}
M.~Z. Oliveira, J.~M. Gomes~da Silva~Jr., and D.~F. Coutinho, ``State feedback
  design for rational nonlinear control systems with saturating inputs,'' in
  \emph{Proceedings of the 2012 {A}merican {C}ontrol {C}onference
  ({ACC})}.\hskip 1em plus 0.5em minus 0.4em\relax IEEE, 2012, pp. 2331--2336.

\bibitem{azizi2018regional}
S.~Azizi, L.~A. Torres, and R.~M. Palhares, ``Regional robust stabilisation and
  domain-of-attraction estimation for {MIMO} uncertain nonlinear systems with
  input saturation,'' \emph{International Journal of Control}, vol.~91, no.~1,
  pp. 215--229, 2018.

\bibitem{Castro_2021}
R.~S. Castro, J.~V. Flores, A.~T. Salton, and J.~M. Gomes~da Silva~Jr,
  ``Controller and anti-windup co-design for the output regulation of rational
  systems subject to control saturation,'' \emph{nternational Journal of Robust
  and Nonlinear Control}, vol.~31, no.~4, pp. 1395--1417, 2021.

\bibitem{SADABADI2016}
M.~S. Sadabadi and D.~Peaucelle, ``From static output feedback to structured
  robust static output feedback: A survey,'' \emph{Annual Reviews in Control},
  vol.~42, pp. 11--26, 2016.

\bibitem{Madeira_2021}
{D. de S. Madeira}, ``Necessary and sufficient dissipativity-based conditions
  for feedback stabilization,'' \emph{IEEE Transactions on Automatic Control,
  {doi:} 10.1109/TAC.2021.3074850}, 2021.

\bibitem{VB:96}
L.~Vandenberghe and S.~Boyd, ``Semidefinite programming,'' \emph{SIAM Review},
  vol.~38, no.~1, pp. 49--95, Mar. 1996.

\bibitem{madeira2020application}
{D. de S. Madeira} and V.~V. Viana, ``An application of {QSR}-dissipativity to
  the problem of static output feedback robust stabilization of nonlinear
  systems,'' \emph{Proceedings of the XIX Brazilian Conference on Automation},
  2020.

\bibitem{willm2}
J.~C. Willems, ``Dissipative dynamical systems part {II}: Linear systems with
  quadratic supply rates,'' \emph{Archive for Rational Mechanics and Analysis},
  vol.~45, no.~5, pp. 352--393, 1972.

\bibitem{hill2}
D.~Hill and P.~Moylan, ``The stability of nonlinear dissipative systems,''
  \emph{IEEE Transactions on Automatic Control}, vol.~21, no.~5, pp. 708--711,
  1976.

\bibitem{safo1}
M.~G. Safonov, \emph{Stability and Robustness of Multivariable Feedback
  Systems}.\hskip 1em plus 0.5em minus 0.4em\relax Cambridge, MA, USA: MIT
  Press, 1980.

\bibitem{meg1}
A.~Megretski and A.~Rantzer, ``System analysis via integral quadratic
  constraints,'' \emph{IEEE Transactions on Automatic Control}, vol.~42, no.~6,
  pp. 819--830, 1997.

\bibitem{Peaucelle_2005}
D.~Peaucelle and D.~Arzelier, ``Ellipsoidal sets for resilient and robust
  static output-feedback,'' \emph{IEEE Transactions on Automatic Control},
  vol.~50, no.~6, pp. 899--904, 2005.

\bibitem{Valmorbida2013}
G.~Valmorbida, S.~Tarbouriech, and G.~Garcia, ``Design of polynomial control
  laws for polynomial systems subject to actuator saturation,'' \emph{IEEE
  Transactions on Automatic Control}, vol.~58, no.~7, pp. 1758--1770, 2013.

\bibitem{RAN2016}
M.~Ran, Q.~Wang, and C.~Dong, ``Stabilization of a class of nonlinear systems
  with actuator saturation via active disturbance rejection control,''
  \emph{Automatica}, vol.~63, pp. 302--310, 2016.

\bibitem{JENNAWASIN2018}
T.~Jennawasin and D.~Banjerdpongchai, ``Design of state-feedback control for
  polynomial systems with quadratic performance criterion and control input
  constraints,'' \emph{Systems \& Control Letters}, vol. 117, pp. 53--59, 2018.

\bibitem{brogl1}
B.~Brogliato, R.~Lozano, B.~Maschke, and O.~Egeland, \emph{Dissipative Systems
  Analysis and Control - Theory and Applications}.\hskip 1em plus 0.5em minus
  0.4em\relax London, UK: Springer-Verlag, 2020.

\bibitem{Mauricio_2001}
M.~C. de~Oliveira and R.~E. Skelton, ``Stability tests for constrained linear
  systems,'' in \emph{Perspectives in robust control}, S.~R. Moheimani,
  Ed.\hskip 1em plus 0.5em minus 0.4em\relax London: Springer London, 2001, pp.
  241--257.

\bibitem{Felipe_2020}
A.~Felipe and R.~C. L.~F. Oliveira, ``An {LMI}-based algorithm to compute
  robust stabilizing feedback gains directly as optimization variables,''
  \emph{IEEE Transactions on Automatic Control}, vol.~66, no.~9, pp.
  4365--4370, 2021.

\bibitem{Lof:04}
J.~L{\"o}fberg, ``{YALMIP}: {A} toolbox for modeling and optimization in
  {MATLAB},'' in \emph{Proceedings of the 2004 IEEE International Symposium on
  Computer Aided Control Systems Design}, Taipei, Taiwan, Sep. 2004, pp.
  284--289, \url{http://yalmip.github.io}.

\bibitem{AA:00}
E.~D. Andersen and K.~D. Andersen, ``The {MOSEK} interior point optimizer for
  linear programming: {A}n implementation of the homogeneous algorithm,'' in
  \emph{High Performance Optimization}, ser. Applied Optimization, H.~Frenk,
  K.~Roos, T.~Terlaky, and S.~Zhang, Eds.\hskip 1em plus 0.5em minus
  0.4em\relax Springer US, 2000, vol.~33, pp. 197--232,
  \url{http://www.mosek.com}.

\bibitem{jenna2021}
T.~Jennawasin and D.~Banjerdpongchai, ``Iterative {LMI} approach to robust
  static output feedback control of uncertain polynomial systems with bounded
  actuators,'' \emph{Automatica}, vol. 123, p. 109292, 2021.

\bibitem{Trofino_2014}
A.~Trofino and T.~Dezuo, ``{LMI} stability conditions for uncertain rational
  nonlinear systems,'' \emph{nternational Journal of Robust and Nonlinear
  Control}, vol.~24, no.~18, pp. 3124--3169, 2014.

\end{thebibliography}

\end{document}